\documentclass{amsart}
\usepackage{graphicx}
\usepackage{amssymb}
\usepackage{graphicx}
\usepackage{amsfonts}
\usepackage{amsmath}
\usepackage{amsthm}
\usepackage{esint}
\usepackage{mathabx}
\usepackage{mathtools}
\usepackage{nccmath}
\usepackage[colorlinks]{hyperref}
\hypersetup{colorlinks=true, pdfstartview=FitV, linkcolor=blue,
            citecolor=blue, urlcolor=blue}  

\usepackage{scrextend,tikz-cd}  
\usepackage{fancyhdr}

\usepackage{enumitem}

\usepackage{tikz}
\usepackage{pgfplots}
\usepackage[mathscr]{euscript}
 \let\mathscr\relax
\usepackage[scr]{rsfso}
\usepackage{multicol}
\theoremstyle{remark}
\newtheorem*{remark}{Remark}

\newcommand{\tensor}{\otimes}

\newcommand{\bR}{\mathbb{R}}

\newcommand\restr[2]{{
  \left.\kern-\nulldelimiterspace 
  #1 
  \vphantom{\big|} 
  \right|_{#2} 
  }}
\theoremstyle{definition}

\newtheorem{theorem}{Theorem}[section]
\newtheorem{lemma}[theorem]{Lemma}

\newcommand{\ep}{\epsilon}

\newtheorem{thm}{Theorem}[section]
\newtheorem{cor}[thm]{Corollary}

\theoremstyle{definition}

\theoremstyle{remark}

\theoremstyle{remark}

\begin{document}

\title[]{Geometric Viewpoint of $L^\infty$ Estimates for Linearized Monge-Amp\`ere Equations}

\author{Arghya Rakshit}
\address{Department of Mathematics, University of Toronto}
\email{\tt arghya.rakshit@utoronto.ca}
\author{Derek Wu}
\address{Department of Mathematics, University of Toronto}
\email{pangolin.wu@mail.utoronto.ca}

\date{}

\begin{abstract}
We prove a Harnack inequality on Monge--Amp\`ere gradient graphs
under suitable geometric bounds, giving us interior $L^\infty$
estimates for the linearized Monge--Amp\`ere equation.
The proof combines Moser iteration on extrensic balls with a gradient estimate for the extrinsic distance function.
The argument avoids the Caffarelli--Guti\'errez Harnack inequality and the engulfing and covering properties of Monge--Amp\`ere sections.
\end{abstract}

\maketitle
\section{Introduction}
The interior regularity theory of Caffarelli and Gutiérrez for the linearized Monge– –Ampère equation \cite{CG96, caffarelli1997properties} is a landmark result in the theory of Monge–Ampère equations. They work with Monge-Ampère sections and assume that the Hessian determinant is bounded away from $0$ and $\infty$. Tian and Wang (see \cite{tianwang2008}) proved the Monge-Ampère Sobolev inequality. In \cite{le2016}, Le studied the Green's function of the linearized Monge--Amp\`ere operator and obtained sharp bounds for it; see also Gu--Le~\cite{gule2026} for a Green's function approach to the regularity of linearized Monge--Amp\`ere equations in divergence form. More recently, Cui, Gu, Le, Wang and Zhou~\cite{cuigulewangzhou2026}
established interior H\"older estimates for linearized Monge--Amp\`ere
equations with bounded divergence-form data in dimensions $n\ge3$,
under two-sided positive bounds for the Hessian determinant.

The geometric study of Monge--Amp\`ere equations goes back to
Calabi~\cite{calabi1958improper}. Kim, McCann and Warren~\cite{kim2010pseudo} showed that smooth optimal transport
graphs are calibrated, hence locally volume-maximizing, spacelike
submanifolds for a pseudo-Riemannian metric built from the cost and the densities. For the quadratic cost and uniform densities this metric is flat. Related geometric methods have been used for Hessian estimates and interior regularity \cite{brendle2024geometric,yuan2023monotonicity,shankar2024regularity}, and for boundary regularity of optimal transport maps \cite{collinstong2025}. 

In this paper, we use this viewpoint to obtain interior $L^\infty$ and Harnack estimates for the linearized Monge--Amp\`ere equation in the constant-density case. We perform Moser iteration on extrinsic balls and use a  Monge-Ampère quasi-distance gradient estimate, following Cheng--Yau~\cite{chengyaumaximal} and
Jost--Xin~\cite{jost2001aspectsglobalgeometryentire} to control the geometry of the extrinsic balls.

The study of Monge-Amp\`ere equations using the extrinsic distance was also studied by Maldonado and Ranabhat in \cite{maldonadoranabhat2020}. In that paper the authors use the quasi-distance along with Caffarelli-Guti\'errez theory to develop improvements, our result demonstrates an alternative formulation of Caffarelli-Guti\'errez theory using geometric methods.

The proof does not use the Caffarelli--Guti\'errez Harnack inequality or the engulfing and covering properties of Monge--Amp\`ere sections. Let us fix some notations before we state our main results:\\
For $n\geq 2$, let $\Omega\subset \mathbb{R}^n,$ be a bounded convex domain, and $u\in C^2(\Omega)$ be a convex solution of
\[
\det D^2u=f\quad \text{in}\quad \Omega.
\]
with $0<\lambda\leq f\leq\Lambda<\infty$.
We write $(u^{ij})=(D^2u)^{-1}$ and $U^{ij}=fu^{ij}$ for the cofactor matrix of $D^2u$. The linearized Monge-Ampère operator  is
\begin{equation}\label{eq:LMA}
    L_uw=U^{ij}w_{ij}=\partial_i (U^{ij}w_j).
\end{equation}
Since our estimates are interior, we impose no boundary condition.   

For $x_0\in\Omega$, define the symmetrized Monge--Amp\`ere
quasi-distance by
\begin{equation}\label{intro:z}
    z_{x_0}(x)=(x-x_0)\cdot\bigl(Du(x)-Du(x_0)\bigr),
    \qquad
    B_R(x_0)=\{x\in\Omega:z_{x_0}(x)<R^2\}.
\end{equation}
This function is nonnegative and vanishes only at $x_0$.
After translating $x_0$ to the origin and subtracting the supporting
affine function, it becomes $z(x)=x\cdot Du(x)$.
We distinguish $B_R(x_0)$ from the intrinsic ball
$B_R^g(x_0)$ for the Hessian metric
\begin{equation}\label{intro:hessian}
    g=\sum_{i,j=1}^n u_{ij}\,dx_i\otimes dx_j,
    \qquad dV_g=\sqrt f\,dx.
\end{equation}
We denote the intrinsic distance by $d$ and the gradient with
respect to the induced metric by $\nabla$. The latter is also
written as $\nabla_u$ when we wish to indicate its dependence on $u$.\\
Below $M=\{(x,Du(x)):x\in \Omega\}$. The main theorem we prove is: \begin{theorem}\label{main}
    Let $w \in C^2(B_R)$ and $\phi \in L^\infty(B_{R})$ and $w$ be
    non-negative. Assume
    \[
        \operatorname{div}_u(A(x)\nabla_u w)=-\phi,
    \]
    where $\operatorname{div}_u$ and $\nabla_u$ are the divergence
    and gradient operators on $(M,u_{ij})$, and $A(x)$ is self-adjoint
    and uniformly elliptic with
    \[
        \lambda_0 I\le A(x)\le \Lambda_0 I.
    \]
    Then
    \[
        \sup_{B_{R/2}} w
        \le C\left(
        \inf_{B_{R/2}} w
        +R^2\|\phi\|_{L^\infty(B_R)}
        \right),
    \]
    where
    \[
        C=C\bigl(n,  \lambda,\Lambda,\lambda_0,\Lambda_0,\operatorname{Vol}(M),\operatorname{diam}(M),
        \lVert H\rVert_{L^\infty}, \lVert \nabla H \rVert_{L^\infty}\bigr).
    \]
\end{theorem}
One can view most of the constants $n, \lambda, \Lambda, \lambda_0, \Lambda_0, \mathrm{Vol}, \mathrm{diam}(M)$ as geometric constants coming from the domain of interest. The dependence on the supremum of $H$ and $\nabla H$ comes from a dependence on the mean curvature coming from the gradient estimate of the Monge-Amp\`ere quasi-distance.\\\\
For a nonnegative solution of $L_uw=-\psi$, we have
\[
    \operatorname{div}_u(\sqrt f\,\nabla_u w)
    =-\frac{\psi}{\sqrt f}.
\]
Taking $A=\sqrt f\,I$ and $\phi=\psi/\sqrt f$ in
Theorem~\ref{main} therefore gives a Harnack inequality,
and hence local $L^\infty$ control, for the linearized
Monge--Amp\`ere equation under the same geometric hypotheses (see Corollary \ref{cor:LMA-harnack}).

 The main geometric ingredient is a gradient estimate for the extrinsic function $z$, which gives suitable cutoff functions and a comparison between $\sqrt z$ and the geodesic distance. We derive a Sobolev inequality from Green's function estimates and a Poincar\'e inequality by integrating along Euclidean
segments. An elementary comparison with Monge--Amp\`ere sections controls the volumes of the extrinsic balls. We then combine
Caccioppoli estimates and Moser iteration with a logarithmic estimate and the John--Nirenberg inequality to obtain the Harnack inequality.\\

The paper is organized as follows. In Section \ref{Sec:Priliminaries}, we introduce the Pseudo-Euclidean geometry of the gradient graph and go through the curvature identities we use later. In Section~\ref{Sec:Main}, we establish the
quasi-distance gradient estimate and distance comparison, discuss
the integral inequalities, and carry out Moser iteration on
extrinsic balls. We then transfer the estimates to quasi-distance
sublevels and apply them to the linearized Monge--Amp\`ere equation. Finally, we discuss some open questions and possible extension of our work.



\section{Preliminaries}\label{Sec:Priliminaries}
\subsection{Monge-Amp\`ere Equation and Section}
Two natural boundary-value problems for the Monge-Amp\`ere equation
are the Dirichlet problem
\begin{align}
    \begin{cases}
        \det D^2u = f & \text{in }\Omega,\\
        u = g & \text{on }\partial\Omega.
    \end{cases}
\end{align}
and the optimal transport problem between domains
\begin{align}
    \begin{cases}
        \det D^2u = f,\\
        \nabla u(\Omega) = \Omega^*.
    \end{cases}
\end{align}
Since we are dealing with interior estimates for this paper we will disregard boundary values.\\

 For $x_0\in\Omega$ and $h>0$, we define the section of $u$ at $x_0$ of height $h$ by
\begin{equation}\label{eq:section}
    S_h(u,x_0)
    :=\{x\in\Omega:
    u(x)<u(x_0)+Du(x_0)\cdot(x-x_0)+h\}.
\end{equation}
These sets are convex. After normalizing $x_0=0$, $u(0)=0$,
and $Du(0)=0$, this becomes $S_h(u,0)=\{x\in\Omega:u(x)<h\}$.
For $t>0$, we write
\[
    tS_h(u,x_0):=x_0+t\bigl(S_h(u,x_0)-x_0\bigr)
\]
for the dilation of the section about $x_0$.

\subsection{Submanifolds of Semi-Euclidean Space}
Let $\bR^{n,n}$ denote the space $\bR^{n +n}$ equipped with the semi-Riemannian metric $\sum_{i=1}^p dx_i \tensor dy_i$ . Given a function $u : \bR^{n} \to \bR$ we can consider the graph $M = \{(x,\nabla u(x)) \mid x \in \bR^{n}\} \subseteq \bR^{n,n}$. Under these coordinates, being spacelike is equivalent to $D^2u>0$. Suppose $u$ is a convex solution to 
\begin{align}
    \det D^2u = f
\end{align}
Using the local coordinates $F: x \in \Omega \subseteq \bR^n \to (x,\nabla u(x)) \subseteq \bR^{n,n}$
\begin{align*}
    H =  \frac1n\Delta_M F = \frac1{2n}\begin{pmatrix}-D^2u^{-1} D \log f\\ D\log f\end{pmatrix}
\end{align*}
Where the $D$ here is the Euclidean gradient and so under the intrinsic metric
\begin{align*}
     \langle H,H\rangle &= -\frac1{4n^2}D\log f\cdot D^2u^{-1} D \log f\\
                    &= - \frac1{4n^2}\frac{u^{ij}f_i f_j}{f^2}
\end{align*}
and we will define $\lVert H \rVert = \sqrt{\frac{1}{4n^2} \frac{u^{ij} f_i f_j}{f^2}}$.
\subsection{Curvature and Monge-Amp\`ere Equation}\label{CurveMA}
We will first do some computation in $\bR^{p,q}$ with a constant semi-Riemannian metric $\langle, \rangle$. If we pick an orthonormal frame $\{e_i, e_\alpha \}_{1
\le i\le p, 1 \le \alpha \le q}$ such that the first $p$ vectors have length 1 and the second $q$ vectors have length $-1$ and an associated dual frame $\{\omega_i, \omega_\alpha\}$ then we can express the second fundamental form of $M$ as 
\begin{gather*}
    \mathrm{II} = \sum_{i,j,\alpha} h^\alpha_{ij} \omega_i \tensor \omega_j \tensor e_{\alpha}
\end{gather*}
And we can express the Ricci tensor as
\begin{align}\label{ricci}
    R_{ij} = \sum_{\alpha=1}^q \left(-h_{ij}^\alpha \left(\sum_k h_{kk}^\alpha\right) + \sum_k h^\alpha_{ik} h^{\alpha}_{kj}\right).
\end{align}
See \cite{chengyaumaximal} or \cite{jost2001aspectsglobalgeometryentire} for a derivation.
\\
We can define the function $z : M \to \bR$ by $z(x) = \langle x,x\rangle$. Given an orthonormal frame field $\{e_1,\ldots e_p, e_{p+1}, \ldots, e_{p+q}\}$ we can compute the derivatives of $z$ as 
\begin{gather}
    z_i(X) = 2\langle e_i, X \rangle\\
    z_{ij} = 2\delta_{ij} + 2\sum_\alpha \langle e_{p + \alpha}, x \rangle h^{\alpha}_{ij}\\
    \Delta z = 2p + 2\sum_{\alpha, i} \langle e_{p + \alpha}, x \rangle h^{\alpha}_{ii}.
\end{gather}
Using the ideas of Jost in \cite{jost2001aspectsglobalgeometryentire}, we can specialize the above computation to an adapted orthonormal frame for $M$. For any $x \in M$ define $\bar x = x - \sum_i^p \langle x, e_i \rangle e_i$
\begin{align}
    e_{p+1} = \frac{\bar x}{\lVert \bar x\rVert } 
\end{align}
Then we iteratively define the $\{e_{p+2}, \ldots e_{p+q}\}$ to be orthogonal to all previous vectors, and then expand around a point to get a local orthonormal frame. In this specific orthonormal frame we get that
\begin{gather}\label{normal}
    z_{ij} = 2\delta_{ij} +2\langle e_{p+1}, x \rangle h^{1}_{ij}\\
    \Delta z = 2p + 2\sum_i \langle e_{p+1}, x\rangle h^1_{ii}.
\end{gather}
In the case that $M$ is a maximal submanifold its mean curvature is 0 so we have $\Delta z = 2p$.
\\
Specializing to the Monge-Amp\`ere equation if we pick $M = \mathrm{Graph} (\nabla u)$ where $u$ solves $\det D^2u = f$ such that $u(0) = 0$ $Du(0) = 0$ we can embed $u$ into $\bR^{n,n} = (\bR^{n + n}, dx \cdot dy)$. Intrinsically $M$ is just the Riemannian manifold $(\Omega, u_{ij})$, and $z$ becomes the quasi-distance for Monge-Amp\`ere equations. Furthermore the Ricci curvature is given as
\begin{align}
    R_{ij}= - n\sum_{\alpha=1}^n h_{ij}^{\alpha} H^\alpha +\sum_{\alpha = 1}^n \sum_{k=1}^n h_{ik}^\alpha h_{kj}^\alpha.
\end{align}
If $f \equiv \mathrm{const}$ then all the mean curvatures vanish and then $R_{ij} \ge 0$.

\section{Main Results}\label{Sec:Main}
\subsection{Gradient Estimate}
Let $u : \Omega \to \bR$ solve $\det D^2 u = f$ for a $C^2$ function $f$. Then we first prove a gradient estimates for the quasi-distance function $z$ on the manifold $M$. We prove this result with a maximal principle argument similar to \cite{chengyaumaximal} Section 3 Theorem 1 and \cite{jost2001aspectsglobalgeometryentire}. We state the gradient estimate for a spacelike submanifold.
\begin{theorem}\label{gradest}
    Let $M \subseteq \bR^{n,n}$ be spacelike such that $0 \in M$. Let $z(X) = \langle X, X\rangle$ be the restriction of the pseudo-norm to $M$. Then for any $c > 0$ and $\alpha \in (0,1)$ such that $\{z < c\}$ is contained in the interior
    \begin{align}
        \sup_{\{z \le \alpha c\}} \frac{|\nabla z|^2}{z} \le C(n, \lVert H\rVert_{\infty} , \lVert \nabla H \rVert_{\infty}, \alpha).
    \end{align}
\end{theorem}
\begin{proof}
    The set $\{z \le c\}$ is compact as proven in \cite[Lemma~100.1]{jost2001aspectsglobalgeometryentire} so this supremum is well-defined. Consider the following function $f(X) = \frac{|\nabla z|^2}{z} \exp(\frac{-m}{c-z})$ where $m >0$ to be chosen later. Then since $z$ is non-negative this function cannot have a boundary maximum (as $z = c$ and the $\exp$ will go to 0), therefore the maximum occurs at an interior point say $q$. Then we have
    \begin{align*}
        \nabla f(q) = 0\\
        \Delta f(q) \le 0.
    \end{align*}
    Expanding and computing gives us the following two equations
    \begin{gather}\label{grad0}
        2\sum_jz_j z_{ij} - w|\nabla z|^2 z_i = 0\\
        2\sum_{i,j}z_{ij}^2 + 2 \sum_{i,j} z_j z_{iji} - w' | \nabla z|^4 - w\left(2\sum_{i,j} z_{ij}z_i z_j + |\nabla z|^2 \Delta z \right) \le 0\label{grad1}
    \end{gather}
    where 
    \begin{align*}
        w = \frac{1}{z} + \frac{m}{(c- z)^2}.
    \end{align*}
    
    By our choice of normal vectors in Equation~\eqref{normal} we have
    \begin{align}
        z_{ij} = 2 \delta_{ij} + 2 \langle X, e_{n+1} \rangle h^{n+1}_{ij}.
    \end{align}
    and so 
    \begin{gather}
        h^{n+1}_{ij} = \frac{z_{ij} - 2\delta_{ij}}{2 \langle X, e_{n+1}\rangle}\\
        \Delta z = 2n + 2n \langle X, e_{n+1}\rangle H^{n+1}.
    \end{gather}
    Furthermore by definition
    \begin{align*}
        \langle e_{n+1}, X\rangle^2 &= \sum_{i=1}^n \langle e_i ,X \rangle^2 - z\\
        &=\frac{|\nabla z|^2}{4} - z.
    \end{align*}
    Rearranging therefore gives
    \begin{align}\label{zidentity}
        z = \frac14 |\nabla z|^2 - \langle X, e_{n+1} \rangle^2.
    \end{align}
    Using Lemma 2 in \cite{yau1975} we get
    \begin{align*}
        \sum_{i,j} z_{ij}^2 \ge \frac{2n-1}{2n-2} \sum_{i} \left(\sum_{j} z_{ij}z_j\right)^2 |\nabla z|^{-2} - \frac1{n-1}(\Delta z)^2.
    \end{align*}
    Substituting Equation~\eqref{grad0} 
    \begin{align*}
        \sum_{i,j} z_{ij}^2 \ge \frac{2n-1}{8n-8}g^2 |\nabla z|^4 - \frac{1}{n-1}\left(2n + n H^{n+1} (|\nabla z|^2 - 4z)^{1/2}\right)^2.
    \end{align*}
    Now for the $z_{i}z_{iji}$ term using Equation~\eqref{ricci}, Equation~\eqref{zidentity} and our adapted coordinates
    \begin{align*}
        \sum_{i,j} z_j z_{iji} &= \sum_{i} z_i (\Delta z)_i + \sum_{\alpha, i,j} R^{}_{ij}z_i z_j\\
        &= \sum_i z_i \left(2n  + 2\sum_{k, \alpha} \langle e_{n+\alpha}, X \rangle h_{kk}^\alpha\right)_i + \sum_{i,j} R_{ij} z_i z_j\\
        &= n\sum_{i, j, \alpha} z_i z_j h_{ij}^\alpha H^\alpha + 2n\sum_{i,j, \alpha}\langle e_{n+\alpha}, X\rangle z_iH^\alpha_i -n\sum_{i,j, \alpha} z_i z_j h_{ij}^\alpha H^\alpha  + \sum_{i,j,k, \alpha} h_{ik}^\alpha h_{kj}^\alpha z_i z_j\\
        &= 2n\sum_{i,j}\langle e_{n+1}, X\rangle z_iH^1_i + \sum_{i,j,k, \alpha} h_{ik}^\alpha h_{kj}^\alpha z_i z_j\\
        &\ge  2n\langle e_{n+1}, x \rangle\sum_{i} z_i H_i^1+  \sum_{i,j,k} h^{1}_{ik} h^{1}_{kj} z_i z_j\\
        &\ge n \sqrt{ |\nabla z|^2 - 4z} \sum_i z_iH_i^{n+1} + 4 - 2g|\nabla z|^2 + \frac14 g^2 |\nabla z|^4.
    \end{align*}
    Plugging everything into our Equation~\eqref{grad1} and simplifying gives
    \begin{equation*}
    \begin{split}
        \left(\frac1{4(n-1)}w^2 - w'\right) |\nabla z|^4 - w(4 + 2n + n |H^{n+1}|\sqrt{|\nabla z|^2 - 4z}) |\nabla z|^2 + \\8 - \frac2{n-1}(2n + n |H^{n+1}|\sqrt{|\nabla z|^2 - 4z})^2 +n \sqrt{ |\nabla z|^2 - 4z} \sum_i z_iH_i^{n+1} \le 0.
    \end{split}
    \end{equation*}
    We need to insure that the coefficient of the $|\nabla z|^4$ term is positive therefore we have
    \begin{align*}
        \left(\frac1{4(n-1)}w^2 - w'\right) = \frac1{4(n-1)}\left(  \frac{1}{z^2} + \frac{2m}{z (c-z)^2} + \frac{m^2}{(c-z)^4}\right) + \frac{1}{z^2} - \frac{2m}{(c-z)^3}.
    \end{align*}
    If we ensure that $\frac{m^2}{4(n-1)(c-z)} \ge 2m$ then this term is always positive which occurs when $m \ge 8(n-1)c$. Then we can write 
    \begin{align*}
        |\nabla z|^4 \le\frac{ w(4 + 2n + n|H||\nabla z|)|\nabla z|^2  + \frac2{n-1} (2n + n |H| |\nabla z|)^2 - n |\nabla z|^2 |\nabla H|}{\frac1{4(n-1)}w^2 - w'}.
    \end{align*}
    In the numerator, the only part that may shoot off towards infinity is $g$ where it is $O(1/z)$ as $z \to 0$ and $O(1/(z-c)^2)$ as $z \to c$. However the denominator is $O(1/z^2)$ and $O(1/(z-c)^4)$. Therefore we can divide by $z^2$ and get
    \begin{align*}
        \frac{|\nabla z|^4}{z^2} &\le A\frac w{z^2 \frac1{4(n-1)}w^2 - w'} + B \frac{|\nabla z|}{z^{1/2}} \frac w{z^{3/2}\frac1{4(n-1)}w^2 - w'} \\&+ C \frac{|\nabla z|^2}{z} \frac w{z \frac1{4(n-1)}w^2 - w'} + D \frac{|\nabla z|^3}{z^{3/2}}\frac{w}{z^{1/2}\frac1{4(n-1)}w^2 - w'}.
    \end{align*}
    Taking $m > 16(n-1)c$ gives
    \begin{align*}
        \frac{w^2}{4(n-1)}- w'\ge \frac{1}{z^2} + \frac{w^2}{8(n-1)}.
    \end{align*}
    Therefore the denominator of each term is bounded.
    \begin{align*}
        \frac{|\nabla z|^4}{z^2} \le A' + B' \frac{|\nabla z|}{z^{1/2}} + C' \frac{|\nabla z|^2}{z} + D \frac{|\nabla z|^3}{z^{3/2}}.
    \end{align*}
    This immediately implies $\frac{|\nabla z|^2}{z} (q) \le C$ for some larger constant $C$, and so $f(q) \le C$ and then for $z \le \alpha c$ with $\alpha <1$ we get 
    \begin{align*}
        \frac{|\nabla z|^2}{z} \le \frac{f(q)}{\inf \{z \le \alpha c\}\exp(\frac{-m}{c-z})} \le C(\alpha).
    \end{align*}
\end{proof}
Now using the discussion in Subsection~\ref{CurveMA} given a solution $\det D^2u = 1$ then the graph of the gradient of $u$ will have zero mean curvature so in particular $H$ and $\nabla H$ therefore we get an estimate on $\frac{|\nabla z|^2}{z} \le C$ on compact subsets away from the boundary.
\subsection{Distance Estimates}
Using this gradient estimate for the quasi-distance $z$ we can prove an equivalence of intrinsic/extrinsic distances in the interior of $M$. Let $d (\cdot, \cdot) : M \times M \to \bR$ be the geodesic distance function on $M$ and assume $0 \in M$. We follow the proof idea of \cite{volumeentropy} where the authors proved a similar distance estimate for maximal submanifolds in spaces of constant negative curvature.
\begin{theorem}\label{distest}
    Assume $x$ is close enough to $0$ such that the minimal distance is achieved via a length minimizing geodesic and assume $x \in \{z (x) < c\}$. There exists a constant $C = C(n,c)$ such that
    \begin{align}
        d(0,x) \le \sqrt{z(x)} \le C(n) d(0, x).
    \end{align}
\end{theorem}
\begin{proof}
    Since the ambient manifold $\bR^{n,n}$ has a constant metric, geodesics are straight lines. Furthermore all points are connected by geodesics (straight lines). Let $\gamma$ be the intrinsic geodesic connecting 0 to the point $x$, then
    \begin{align}
        \sqrt{z(x)} - 0 &= \int_0^1 \langle \nabla \sqrt{z}\mid_{\gamma (t)}, \dot \gamma(t) \rangle dt\\
        &\le \int_{[0,1]} \lVert \nabla \sqrt{z}\rVert_{L^\infty} \lVert \dot \gamma \rVert dt\\
        &\le Cc \ell(\gamma)\\
        &\le C' d(0,x).
    \end{align}
    Now for the other direction we know that $d(0,x)$ is given as the length minimizing geodesic therefore
    \begin{align*}
        d(0,x) &\leq \int_0^1\sqrt{x^T D^2u(tx)x}\,dt\\
        &\leq \left(\int_0^1x^T D^2u(tx)x\,dt\right)^{1/2}\\
        &=\sqrt{x\cdot(Du(x)-Du(0))}
        =\sqrt{z(x)}.
    \end{align*}
\end{proof}

\begin{remark}
    By a change of coordinates, define $z_p : M \to \mathbb{R}$ to be $z_p(X) = \langle X-p, X-p\rangle$ the pseudo-norm distance to a point $p$, then we have the same gradient estimate and distance estimate with constants depending on the same things.
\end{remark}
\subsection{$L^\infty$ Estimates and Sobolev Inequality}
We will now prove an $L^\infty$ for the linearized Monge-Amp\`ere equation. The primary method is to use Moser-Harnack iteration, to do this we require a Sobolev inequality and a Caccioppoli-type inequality.\\
In this section we assume $M = \mathrm{Graph}(Du) \subseteq \bR^{n,n}$ with $\det D^2u = f$. We make the following assumptions on our manifold $M$
\begin{enumerate}
    \item The Ricci curvature $R_{ij} \ge (n-1)k$ for $k < 0$
    \item The mean curvature $H$ satisfies $H$ and $\nabla H$ are bounded.
\end{enumerate}
In terms of the PDE $\det D^2u = f$ the first constraint simply corresponds to $H$ being bounded as well. The requirements for $H$ and $\nabla H$ to be bounded correspond to $\nabla f$ and $\nabla^2f$ being bounded where $\nabla$ is the Levi-Civita connection of the manifold $M$. Although it is rather unclear what $f$ satisfy these assumptions, the simple example of $\det D^2u = 1$ satisfies them. Additionally the curvature bounds on $M$ on needed only for the gradient and distance estimates. The Monge-Amp\`ere Sobolev and Poincar\`e inequality can be derived entirely from the structure of the equation and standard methods without referencing Caffarelli-Gutierrez theory.
\\
Given our curvature constraints we could also derive Sobolev inequalities from Ricci curvature bounds using Riemannian geometry. However instead we prove it using an $L^p$ bound on the Green's function which is well-known and standard, we provide a proof for completeness.\\\\
The following proof will be stated for a general divergence form uniformly elliptic PDE on some domain $\Omega$ in $\bR^n$. Our proof of the Sobolev inequality is inspired by the
Green's function approach of Guo, Phong, Song and Sturm~\cite[Section~4]{GPSS23}.
\begin{theorem}\label{sobolev}
Consider an uniformly elliptic operator $L[u] = -\mathrm{div} \left(A(x) \nabla u \right)$ where $A(x)$ is symmetric positive semi-definite on a domain $\Omega$. Assume $G(x,y)$ is its Green's function (see \cite{gruterwidman1982}) and satisfies
\begin{align*}
    \sup_{x} \lVert G(x,\cdot) \rVert_{L^p(\Omega, dx)} \le M.
\end{align*}
Then we have for any function $\phi \in C^2(M)$ such that $\phi \equiv 0$ on $\partial \Omega$
\begin{align*}
    \lVert \phi \rVert_{L^{2 p}}^2 \le M \int A(x) \nabla \phi \cdot \nabla \phi dx.
\end{align*}
\end{theorem}
\begin{proof}
    Define the operator $T(f) = \int G(x,y) f(y) dy$ then
    \begin{align}
        \lVert T(f) \rVert_{L^\infty} \le \lVert G(x,\cdot) \rVert_{L^p} \lVert f \rVert_{L^q} = M \lVert f\rVert_{L^q}.
    \end{align}
    For $q = \frac{p}{p-1}$. Additionally
    \begin{align*}
        \lVert T(f) \rVert_{L^p} &= \left(\int \left( \int G(x,y) f(y) dy\right)^pdx\right)^{1/p}\\
        &\le\int\left(\int G(x,y)^p f(y)^p dx \right)^{1/p}dy\\
        &= \int |f(y)|\left(\int G(x,y)^p dx \right)^{1/p}dy\\
        &\leq M \lVert f\rVert_{L^1}.
    \end{align*}
    Therefore by Riesz-Thorin interpolation
    \begin{align*}
        \lVert T(f)\rVert_{L^{2p}} \le M \lVert f \rVert_{L^{\frac{2p}{2p-1}}}.
    \end{align*}
    Now for a smooth compactly supported function $\phi$ we have
    \begin{align*}
        \left|\int \phi fdx\right|^2 &= \left|\int \phi(x) \int G(x,y) L[f](y) dy dx\right|^2\\
        &= \left|\int  Tf L[\phi] dy dx\right|^2\\
        &\le \int Tf L[Tf]dx \int \phi L[\phi] dx\\
        &= \int f Tf dx \int \phi L[\phi] dx\\
        &\le \lVert f \rVert_{L^{\frac{2p}{2p-1}}} \lVert T f \rVert_{L^{2p}} \int \phi L[\phi] dx\\
        &\le M\lVert f \rVert_{L^{\frac{2p}{2p-1}}}^2 \int \phi L[\phi] dx.
    \end{align*}\\
    Now take $f = \frac{\phi^{2p-1}}{ \lVert \phi \rVert_{L^{2p}(dx)}^{2p-1}}$ so that $\lVert f\rVert_{L^{\frac{2p}{2p-1}}(dx) } = 1$. Then plugging this in gives
    \begin{align*}
        \left(\int \phi^{2p}dx\right)^2\frac1{\lVert \phi \rVert_{L^{2p}(dx)}^{4p-2}} = \lVert \phi\rVert_{L^{2p}(dx)}^2 
    \end{align*}
    Therefore
    \begin{align*}
        \lVert \phi\rVert_{L^{2p}(dx)}^2  \le M \int A(x) \nabla \phi \nabla \phi dx
    \end{align*}
\end{proof}
Consider a convex domain $\Omega$, and let $L$ be the linearized Monge-Amp\`ere equation expressed as $L[\phi] =  U^{ij} \phi_{ij}$, uniformly bounded away from 0 and $\infty$. Consider the following PDE.
\begin{align*}
    \begin{cases}
        L [\phi] = g &\text{in } \Omega\\
        \phi \equiv 0 &\text{on } \partial \Omega
    \end{cases}
\end{align*}
By the Alexandroff-Bakelman-Pucci estimate (see \cite[Lemma~9.1]{GT})
 \begin{align*}
    \lVert \phi \rVert_{L^\infty} \le C(n) \mathrm{diam}(\Omega) ||g||_{L^n(\Omega).}
\end{align*}

Therefore using Green's representation theorem
\begin{align*}
    \sup_x \int G(x,y) g(y)dy \le C\mathrm{diam}(\Omega)||g||_{n}
\end{align*}
Therefore we have
\begin{align*}
    \sup_x \lVert G(x, \cdot) \rVert_{L^{\frac{n}{n-1}}} \le C \mathrm{diam}(\Omega).
\end{align*}
So we have a Sobolev inequality for $\phi \in C^2_0(M)$ (functions vanishing on the boundary) of the form
\begin{align*}
    \lVert \phi \rVert_{L^{\frac{2n}{n-1}}}^2 \le C(n)\mathrm{diam}(\Omega) \int_\Omega U^{ij} \phi_i \phi_j dx
\end{align*}

Since $f$ has uniform bounds and the Riemannian volume form is $\sqrt{f}dx$, by inflating the constant we can move back to our manifold $(M, u_{ij})$, this gives us the desired Sobolev inequality
\[
 \|\phi\|_{L^{\frac{2n}{n-1}}(\Omega,dV)}^2
    \le C(n,\lambda,\Lambda)\operatorname{diam}(\Omega)
    \int_\Omega |\nabla\phi|_g^2\,dV.
\]

For the $L^\infty$ estimate we will also need a Poincar\'e inequality the manifold $M$. It is also of independent interest that a Poincar\`e inequality holds depending on the Monge-Amp\`ere quasi-distance. We write $\fint_{U} d\mu = \frac1{\mu(U)} \int \phi d\mu$ where the volume is computed using whatever measure we are integrating against, and we write $\phi_{U, \mu} = \frac1{\mu(U)} \int \phi d\mu$. If it is clear what measure we are integrating against we will just write $\phi_{U}$.
\begin{lemma} \label{poincare}
For any convex set $U$ and $\phi \in C^{\infty}(U) \cap L^2(U)$ there exists a constant $C = C(n, \lambda, \Lambda)$ such that
\begin{align*}
    \int_{U} |\phi - \phi_U|^2 dV \le C R^2 \int_U |\nabla \phi|^2 dV.
\end{align*}
where $R^2 = \sup_{x,y \in U} \langle x-y, Du(x)-Du(y) \rangle. $
\end{lemma}
\begin{proof}
    For the proof of Poincar\'e inequality we will work in coordinates instead of on the manifold. Therefore integration is with respect to Lebesgue measure, $x\cdot y$ represents Euclidean inner product and $|\cdot |$ is the standard norm. By Jensen's inequality
    \begin{align*}
        \int_{U} |\phi - \phi_U|^2 dx &\le \int_U \fint_U |\phi(x) - \phi(y)|^2 dx dy.
    \end{align*}
    Consider $x,y \in \Omega$. Let $\gamma(t) = tx + (1-t)y$ applying Cauchy-Schwarz twice gives
    \begin{align*}
        |\phi(x) - \phi(y)|^2 &\le \left(\int_0^1 |D\phi(\gamma(t)) \cdot \gamma'(t)|dx\right)^2\\
        &\le\left(\int_0^1 (D\phi^{T}(D^2u(\gamma(t)))^{-1}D\phi^{1/2} (\gamma'(t)D^2u(\gamma(t))\gamma'(t))^{1/2} dt\right)^2\\
        &\le \int_0^1  u^{ij} \phi_i \phi_j dt \int_0^1 \gamma'(t)D^2u(\gamma(t))\gamma'(t)dt\\
        &= \int_0^1  u^{ij} \phi_i \phi_j dt (Du(y) - Du(x)) \cdot(y-x).
    \end{align*}
    Now define $R^2 = \sup_{x,y} (Du(y) - Du(x))\cdot (y-x)$ then 
    \begin{align*}
        |\phi(x) - \phi(y)|^2 \le R^2\int_0^1 u^{ij}(\gamma(t)) \phi_i \phi_j dt.
    \end{align*}
    Therefore using Fubini and symmetry (and writing $\mathrm{vol}(U)$ for Euclidean Lebesgue measure of $U$)
    \begin{align*}
        \int_U \fint_U |\phi(x) - \phi(y)|^2 dx dy &\le \frac {R^2}{\mathrm{vol}(U)}\int_U \int_U\int_0^1 u^{ij}\phi_i \phi_j (tx + (1-t)y)dtdxdy \\
        &= \frac {R^2}{\mathrm{vol}(U)}\int_0^1\int_U \int_U u^{ij}\phi_i \phi_j (tx + (1-t)y)dxdy dt\\
        &= \frac{2R^2}{\mathrm{vol}(U)}\int_{1/2}^1\int_U \int_U u^{ij}\phi_i \phi_j (tx + (1-t)y)dxdy dt.
    \end{align*}
    Then for fixed $y$ take $z = tx + (1-t)y$ so $dz = t^{n} dx$
    \begin{align*}
        \frac{2R^2}{\mathrm{vol}(U)}\int_{1/2}^1\int_U \int_U& u^{ij}\phi_i \phi_j (tx + (1-t)y)dxdy dt\\
        &= \frac{2R^2}{\mathrm{vol}(U)}\int_{1/2}^1 t^{-n}\int_U \int_{tU + (1-t)y} u^{ij}\phi_i \phi_j (z)dzdy dt\\
        &\le 2{R^2} \frac{2^{n-1} - 1}{n-1}\int_Uu^{ij}\phi_i \phi_j (z)dz.
    \end{align*}
    Therefore we have shown 
    \begin{align*}
        \int_{U} |\phi - \phi_U|^2 dx \le C_n {R^2} \int_U u^{ij} \phi_i \phi_j dx.
    \end{align*}
    Where $C_n$ only depends on dimension. In the case of $\det D^2u = f$ where $0 < \lambda \le f \le \Lambda < \infty$, then the induced measure is $\sqrt{f}dx$ which differs only by constant multiples therefore we can absorb it into the constant $C(n, \lambda, \Lambda)$. Furthermore moving back to the manifold structure the term $R^2$ represents $\sup_{x,y \in \Omega} \langle x-y, x-y \rangle$ which is just the diameter using the restricted pseudo-norm on $M$. Therefore for any convex set $U$ we have
    \begin{align*}
        \int_U (\phi - \phi_{U, dV})^2 dV \le C(n, \lambda, \Lambda) R^2 \int_{U} |\nabla \phi|^2 dV.
    \end{align*}
\end{proof}
\begin{remark}
    One can do the same trick for any positive semi-definite matrix $a^{ij}$ which would give a proof of a Poincar\'e inequality for linearized elliptic equations with a constant depending on some generalization of $z$ to other equations.
\end{remark}

\subsection{$L^\infty$ Estimates}
Using the gradient estimate and Sobolev inequality we can now proceed with de Giorgi-Nash-Moser iteration to prove an $L^\infty$ estimate for solutions of divergence form elliptic equations on $M$. \\

Recall $B_{r_1} = \{z \le r_1^2\}$ and let $r < R$.  First we use the gradient estimate to construct cutoff functions adapted to $z$:
\begin{lemma}\label{cutoff}
    There exists a a smooth function $\eta: B_R \to [0,1]$  such that 
    \begin{align}
        \eta(x) = \begin{cases}
            0 & z(x) > R^2\\
            1 & z(x) < r^2\\
            \text{Between $0$ and $1$} & r^2 \le z(x) \le R^2
        \end{cases}
    \end{align}
    and such that there exists a constant $C = C(n, \lVert H\rVert_{L^\infty}, \lVert \nabla H \rVert_{L^\infty})$
    \begin{align}
        \lVert \nabla \eta\ \rVert_{L^\infty} \le  \frac{C}{R - r}.
    \end{align}
\end{lemma}
\begin{proof}
    Take $\eta(x) = \frac{R - \sqrt{z}}{R- r} $ on $B_R \setminus B_R$ and extend by 0 and 1 everywhere else. Then using Theorem~\ref{gradest} (and absorbing all the constants)
    \begin{align*}
        |\nabla \eta| = \frac{|\nabla z|}{2\sqrt{z}(R - r)} \le \frac{C}{R - r}.
    \end{align*}
    
\end{proof}
Now we can proceed with de Giorgi-Nash-Moser iteration on $M$ using the quasi-distance $z$ to get an a priori estiamte. We write $B_R = \{x \mid z(x) < R^2 \}$.

\begin{theorem}
    Let $w \in C^2(B_R)$ and $\phi \in L^\infty(B_{R})$ and $w$ is non-negative. Assume $\mathrm{div}_u(A(x) \nabla_u w) = -\phi$ where $\mathrm{div_u}$ and $\nabla_u$ are the divergence and gradient operator for $(M_, u_{ij})$ and $A(x)$ is a self-adjoint uniformly elliptic matrix such that $\lambda_0 \le A(x) \le \Lambda_0$,
    \begin{align}
        \sup_{B_{R/2}} w \le C (\inf_{B_{R/2}} w  + R^2 \lVert \phi\rVert_{ L^\infty(B_R) }).
    \end{align}
    Where $C(n,\lambda_0, \Lambda_0, \mathrm{Vol(M)}, \mathrm{diam}(M), \lVert \mathrm{H} \rVert_{L^\infty}, \lVert \mathrm{\nabla H} \rVert_{L^\infty})$
\end{theorem}
\begin{proof}
    Let $k = R^2 \lVert \phi\rVert_{L^\infty(B_R)} +\ep$ for some $\ep$ chosen later and write $v = w + k$. All constants used will be independent of $\ep$. Pick a compactly supported test function $\eta$ supported on the set $B_R$. Using integration by parts while testing against $\eta^2 v^\beta$ for some $\beta \in \bR$ gives
    \begin{align*}
        \beta \int_{B_R} \eta^2 v^{\beta - 1}\langle A(x)\nabla v, \nabla  \rangle \;dV + 2 \int_{B_R}\eta v^\beta \langle A(x)\nabla v, \nabla \eta \rangle \;dV &= \int_{B_R} \phi \eta^2 v^\beta dV.
    \end{align*}
    Using Young's inequality, ellipticity and since $|\phi| \le R^{-2} v$
    \begin{align}
        |\beta|\lambda_0 \int_{B_R}  \eta^2 v^{\beta - 1}|\nabla v|^2dV &\le 2 \Lambda_0\int_{B_R} \eta v^\beta|\nabla v| |\nabla \eta|dV + R^{-2}\int_{B_R} \eta^2 v^{\beta + 1}dV\\
        &\le  \frac{|\beta| \lambda_0}{2}\int_{B_R} \eta^2 v^{\beta - 1}|\nabla v|^2 dV+\frac{2\Lambda_0^2}{|\beta| \lambda_0}\int_{B_R}v^{\beta+1} |\nabla \eta|^2 + R^{-2} \int_{B_R} \eta^2 v^{\beta + 1}dV\\
         \int_{B_R}  \eta^2 v^{\beta - 1}|\nabla v|^2dV &\le  \frac{4\Lambda_0^2}{|\beta|^2\lambda_0^2}\int_{B_R}v^{\beta+1} |\nabla \eta|^2 + \frac{2}{|\beta|R^2\lambda_0} \int_{B_R} \eta^2 v^{\beta + 1}dV.
    \end{align}
    Suppose $\beta \not=-1,0$. Then we set $\kappa= \frac{n}{n-1}$ and $A_\beta = 1 + \frac{(\beta + 1)^2}{\beta^2} + \frac{(\beta + 1)^2}{|\beta|}$. Then substituting our previous inequality into $\int |\nabla (\eta v^{(\beta + 1)/2})|$ gives 
    \begin{align*}
        \int_{B_R} \left|\nabla (\eta v^{(\beta + 1)/2}) \right|^2 &= \int_{B_R} |\nabla \eta|^2 v^{\beta +1} + 2\langle \nabla \eta, \nabla v \rangle \eta v^\beta + \eta^2 \left(\frac{\beta + 1}{2} \right)^2v^{\beta-1} |\nabla v|^2 dV\\
        &\le \int_{B_R} |\nabla \eta|^2 v^{\beta +1} + |\nabla \eta|^2  v^{\beta + 1} + \eta^2|\nabla v|^2 v^{\beta-1} + \left( \frac{\beta +1}{2}\right)^2 \eta^2 v^{\beta - 1} |\nabla v|^2 dV\\
        &\le(n, \lambda_0, \Lambda_0) C A_\beta \int_{B_R} (|\nabla \eta|^2 + R^{-2}\eta^2) v^{\beta + 1}dV.
    \end{align*}
    Specifically the constant $C$ does not depend on $\beta$. Now pick $R/2 < r < s < 2R/3$ and use Lemma~\ref{cutoff} to construct a function supported in $B_r$ and 0 outside $B_s$. Then we get
    \begin{align*}
        \int_{B_R} \left|\nabla (\eta v^{(\beta + 1)/2})\right|^2dV\le   \frac{CA_\beta}{(s-r)^2}\int_{B_s} v^{\beta + 1}dV.
    \end{align*}
    We apply Theorem~\ref{sobolev} to get (and inflate our constant again)
    \begin{align}
        \left(\int_{B_r}  |v|^{(\beta+1)\frac{n}{n-1}})\right)^{\frac{n-1}{n}} \le \frac{CA_\beta}{(s-r)^2}\int_{B_s} v^{\beta + 1}dV.
    \end{align}
    By  Theorem~\ref{distest} we know that the sets $B_r$ are comparable at all scales to geodesic balls, and since the volume form on $M$ is just $fdx$, where $dx$ is Euclidean Lebesgue measure , and $f$ is bounded away from 0 and infinity we get the averaged inequality
    \begin{align}\label{moser1}
        \left(\fint_{B_r}  |v|^{(\beta+1)\frac{n}{n-1}})\right)^{\frac{n-1}{n}} \le \frac{CA_\beta \mathrm{Vol}(B_s)}{(s-r)^2 \mathrm{Vol}(B_r)^{(n-1)/n}}\fint_{B_s} v^{\beta + 1}dV.
    \end{align}
     We now need to estimate the volume of the pseudo-balls. By \cite{collinstong2025} Proposition 2.6, we have that for a point $x_0$ and $r >0$
    \begin{align*}
        \frac12 S_{r^2}(u,x_0) \subseteq B_r(x_0) \subseteq S_{r^2}(u,x_0),
    \end{align*}
    where $S_{r^2}(u,x_0)$ is a section of a convex function at $x_0$ of height $r^2$. By convexity and uniform bounds on $\det D^2u$ we have $c(n, \lambda, \Lambda) r^n \le \mathrm{Vol}( S_{r^2}(u,x_0) \le C(n, \lambda, \Lambda) r^n$ therefore the extrinsic balls also have approximately $r^n$ volume as well. Therefore this gives us
    \begin{align}
        \left(\fint_{B_r}  |v|^{(\beta+1)\frac{n}{n-1}})\right)^{\frac{n-1}{n}} \le \frac{CA_\beta R}{(s-r)^2} \fint_{B_s} v^{\beta + 1}dV.
    \end{align}
    Now assume $\beta >-1$. We may pick a $\delta$ only depending on $n$ such that $\delta \kappa^j = \delta (\frac{n}{n-1})^j \not= 1$ for all $j > 0$. Then set
    \begin{align*}
        p_j = \delta \kappa^j \hspace{20pt} \beta_j = p+j - 1\hspace{20pt}r_j = \frac{R}{2} + \frac{R}{6}2^{-j}.
    \end{align*}
    Now take $\frac{1}{p_j}$ power on both sides and using an $L^p$ norm with an averaged integral gives
    \begin{align*}
        \lVert v \rVert_{L^{p_{j+1}}(B_{r_{j+1}}, dV)} \le (C \frac1R 4^j A_{\beta_j})^{1/p_{j}} \lVert v \rVert_{L^{p_j}(B_{r_j}, dV)}.
    \end{align*}
    Iterating gives
    \begin{align*}
        \lVert v \rVert_{L^{p_{j+1}}(B_{r_{j+1}}, dV)} \le \left(\prod_{k=1}^j(C\frac1R 4^k A_{\beta_k})^{1/p_{k}} \right)\lVert v \rVert_{L^{p_j}(B_{r_j}, dV)}.
    \end{align*}
    Taking logarithm gives $\sum_{k=1}^j \frac{1}{
    \delta \kappa^j}\left( \log C + k\log 4 + \log A_{\beta_k} + \log 1/R\right)$ in the coefficient. Since the inside part grows linearly this summation converges for fixed $R$. Therefore we can take a limit and get 
    \begin{align}
        \sup_{B_{R/2} } v\le C \lVert v \rVert_{L^{\delta}(B_{2R/3}, dV )}.
    \end{align}
    Now assume $\beta < - 1$. Take $\beta_j = -\delta\kappa^{j}$ and $\beta_j =p_j -1$. We do the same thing for $p_j < 0$ which flips the inequality 
    \begin{align*}
        \lVert v \rVert_{L^{p_{j+1}}(B_{r_{j+1}}, dV)} \ge (C \frac1R4^j A_{\beta_j})^{1/p_{j}} \lVert v \rVert_{L^{p_j}(B_{r_j}, dV)}.
    \end{align*}
    Doing the same thing as before
    \begin{align*}
        \inf_{B_{R/2}} v\ge C \lVert v \rVert_{L^{\delta}(B_{2R/3}, dV )}.
    \end{align*}
    Our goal is to now prove that for small $\delta$ there exists a constant $C$ such that
    \begin{align}\label{goal}
        \left(\fint_{B_{2R/3}} v^\delta \right)\left( \fint_{B_{2R/3}} v^{-\delta}\right) \le C.
    \end{align}
    Take $h = \log v$ and pick a cutoff function $\eta$ using Lemma~\ref{cutoff} such that $\eta \equiv 1$ on $B_r$ and supported on $B_{4r/3}$ with $4r/3<R$. Then $|\nabla \eta| \le \frac{C}{r}$. Then testing our equation against $\eta^2/v$ gives 
    \begin{align*}
        \int_{B_R}  \langle A \nabla v, \nabla (\eta^2/v) \rangle dV = \int_{B_R}\frac{\phi}{v}\eta^2 dV
    \end{align*}
    Rearranging gives
    \begin{align*}
        \int_{B_r} \eta^2\langle A \nabla h, \nabla h\rangle dV &= \int_{B_R} 2\langle A \nabla h, \nabla \eta \rangle - \frac{\phi}{v}dV\\
        \lambda \int_{B_r} \eta^2|\nabla h|^2 dV &= \int_{B_R}2 \Lambda \eta \langle \nabla h, \nabla \eta \rangle - \frac{\phi}{v}dV\\
        &\le \frac{\lambda}{2}\int_{B_R} \eta^2 |\nabla h|^2 dV + \frac{2\Lambda}{\lambda} \int_{B_R}|\nabla \eta|^2 dV + R^{-2} \int_{B_R} \eta^2 dV
    \end{align*}
    Therefore
\begin{align}\label{beta-1}
        \int_{B_R}\eta^2 |\nabla w|^2 dV &\le \frac{4\Lambda}{\lambda^2} \int_{B_R} |\nabla \eta|^2dV + \frac{2}{\lambda R^2}\int_{B_R} \eta^2 dV\\
        &\le Cr^{-2} \mathrm{Vol}(B_r)
    \end{align}
    Now for any $y$ and sufficiently small $r$ consider the the convex hull of pseudo-norm balls $\tilde B_r(y) = \mathrm{conv}(\{x \mid \langle y-x, y-x \rangle \le r^2\})$. We want to apply our Poincar\'e inequality, but to do this we first show that the pseudo-norm diameter of $\tilde B_{r}(y)$ is bounded by a constant multiple of $r$. 
    \begin{lemma}
        \begin{align*}
            \mathrm{diam}(B_r(0)) \le C r
        \end{align*}
    \end{lemma}
    \begin{proof}
        Write $\pi : M \to \Omega$ be the projection of the graph onto its domain. Then pick $p_1, \ldots, p_\ell \in \pi(B_r(0))$ and a convex combination $t_1, \ldots, t_\ell$ and write $q = t_1p_1 + \ldots + t_\ell p_\ell$. Then 
        \begin{align*}
            \sum_{i=1}^\ell t_i (q- p_i)\cdot (Du(q) - Du(p_i)) &=\sum_i -t_i p_i \cdot (Du(q) - Du(p_i)) - t_i q \cdot Du(p_i) + q \cdot Du(q)\\
            &=\sum_{i} t_i p_i \cdot Du(p_i) - t_i q \cdot Du(p_i)\\
            &= \sum_{i,j} t_it_j p_i \cdot Du(p_i) - \sum_{i,j}t_i t_j p_j \cdot Du(p_i)\\
            &= \frac12\sum_{i,j} t_i t_j (p_i - p_j) \cdot (Du(p_i) - Du(p_j)).
        \end{align*}
        Identifying $p, q$ with points on the graph $M$ we get
        \begin{align*}
            z(q,p_i) \le \frac12 \sum_{i,j} t_i t_j z(p_i, p_j).
        \end{align*}
        By the distance estimates in theorem \ref{distest} we have $\sqrt{z(p_i, p_j)} \le C(d(0, p_i) + d(0, p_j)) \le 2Cr$. Therefore $z(q,p_i) \le 2C^2 r^2$, and using the distance estimates again this implies $d(0,q) \le d(0, p_i) + d(p_i, q) \le (1 + \sqrt{2}C)r$. Therefore the convex hull of $B_r(0)$ is contained in a larger pseudo-norm ball of radius $2C(1 + \sqrt{2}C)r$. But using our equivalence of distances again, pseudo balls of radius $R$ have diameters bound by some constant multiple of $R$.
    \end{proof}
    Now we can safely apply Lemma \ref{poincare} with an inflated constant to get
    \begin{align}
        \left(\int_{\tilde B_r(y)} (w - w_{\tilde B_r(y)})^2dV\right)^{\frac12} \le  Cr \left(\int_{\tilde B_r(y)}|\nabla w|^2dV\right)^{1/2}.
    \end{align}
    Using Equation~\eqref{beta-1} and the previous result gives
    \begin{align*}
        \fint_{\tilde B_{r}(y)} (w - w_{\tilde B_{r}(y)})^2dV\le C r^2 \mathrm{Vol}(B_r) r^{-2}\frac1{Vol(B_r)} \le C.
    \end{align*}
    Since the volume form $dV = fdx$ for $f$ bounded away from 0 and infinity, and the balls using the pseudo-norm $z$ are equivalent at all scales to geodesic balls, we get $w \in BMO(B_{2R/3})$, we can apply Jonh-Nirenberg inequality to $w$ to get
    \begin{align}
        \fint B_{2R/3} \exp(\delta|w - w_{B_{2R/3}}|) \le C.
    \end{align}
    Since $v = e^w$, using sufficiently small $\delta$ we have now proved Equation~\eqref{goal}. Therefore
    \begin{align*}
        \sup_{B_{R/2}} v \le C \lVert v\rVert_{L^\delta(B_{2R/3})} \le C\lVert v\rVert_{L^{-\delta}(B_{2R/3})} \le C\inf_{B_{R/2}} v.
    \end{align*}
    Now taking $\ep \to 0$ and substituting back $w$ gives the desired Harnack inequality.
\end{proof}
\begin{remark}
  There are also many different variation of Moser-Harnack inequality, for example if we assume positivity of $\phi$ we can weaken the assumption to $\lVert \phi\rVert_{L^q}$ for $q > n$ and only require $\Delta v \ge - \phi$. 
\end{remark}
Finally now consider a solution to the linearized Monge-Amp\`ere equation
\begin{align*}
    U^{ij} w_{ij} = -\psi
\end{align*}
where $U^{ij} = f u^{ij}$. Then we have for any test function
\begin{align*}
    \int_M \mathrm{div}_u( \sqrt{f} \nabla_u w) \eta dV&=-  \int_M  \sqrt{f }\langle \nabla_u w, \nabla_u \eta \rangle dV\\
    &=-\int_M u^{ij} w_i \eta_j fdx\\
    &= -\int_M \frac{\psi}{\sqrt{f}}\eta dV.
\end{align*}
Therefore any solution to a linearized Monge-Amp\`ere equation satisfies a uniformly elliptic divergence form equation on $M$, so we can apply our result to prove a Harnack inequality. Standard methods gives $C^\alpha$ regularity, which immediately gives $C^{1,\alpha}$ regularity for the Monge-Amp\`ere equations.

\begin{cor}\label{cor:LMA-harnack}
Assume that $M$ and $B_R$ satisfy the geometric hypotheses of
Theorem~\ref{main}. Let $u$ be a convex solution of
$\det D^2u=f$, where $0<\lambda\le f\le\Lambda<\infty$, and
write $U^{ij}=fu^{ij}$. If $w\in C^2(B_R)$ is nonnegative and
satisfies
\[
    U^{ij}w_{ij}=-\psi \quad \text{in } B_R,
    \qquad \psi\in L^\infty(B_R),
\]
then
\begin{equation}\label{eq:LMA-harnack}
    \|w\|_{L^\infty(B_{R/2})}
    \le C\left(
        \inf_{B_{R/2}}w
        +R^2\|\psi\|_{L^\infty(B_R)}
    \right),
\end{equation}
where $C$ depends on $\lambda,\Lambda$ and the geometric
quantities appearing in Theorem~\ref{main}.
In particular, every nonnegative solution of $U^{ij}w_{ij}=0$
satisfies
\[
    \sup_{B_{R/2}}w\le C\inf_{B_{R/2}}w.
\]
\end{cor}

\section{Future directions}\label{Sec:Open}
 The geometric approach developed here suggests possible extensions
to other nonlinear elliptic equations. In particular, it would be
interesting to obtain similar Harnack estimates for the linearized
equations arising in optimal transport with nonquadratic costs.
A natural question is which assumptions on the cost and the
densities allow the geometric arguments used here to carry over. In particular, it would be interesting to understand whether
the Ma--Trudinger--Wang condition (see \cite{MTW05}) on the cost allows analogous Harnack estimates to be obtained by this geometric approach.

Another direction is to better understand the geometric hypotheses
in the nonconstant-density case. It would be useful to identify
natural analytic conditions on the density and the potential that
ensure the curvature bounds required by our argument. More generally,
one may ask whether these hypotheses can be weakened while retaining
a geometric proof of the Harnack inequality.


\section*{Acknowledgments}
The authors would like to thank Professor Freid Tong for several helpful discussions related to this paper.

\end{document}